\documentclass{amsart}
\usepackage{color,soul}
\usepackage[margin=1in]{geometry}
\usepackage{amssymb, amsmath}
\usepackage{float}
\usepackage{relsize}
 \usepackage[abs]{overpic}	
\usepackage{enumerate}

\usepackage{amsthm}

\usepackage{amssymb, amsmath}
\usepackage{float}

\usepackage[colorlinks = true,
            linkcolor = red,
            urlcolor  = red,
            citecolor = red,
            anchorcolor = red]{hyperref}
\usepackage{amsrefs}

\newtheorem{theorem}{Theorem}[section]
\newtheorem{proposition}[theorem]{Proposition}
\newtheorem{corollary}[theorem]{Corollary} 
\newtheorem{lemma}[theorem]{Lemma}
\newtheorem{remark}[theorem]{Remark}
\newtheorem{definition}[theorem]{Definition}
\newtheorem{example}[theorem]{Example}

\DeclareRobustCommand{\rchi}{{\mathpalette\irchi\relax}}
\newcommand{\irchi}[2]{\raisebox{\depth}{$#1\chi$}}

\begin{document}

   \title{Quandles with one non-trivial column}
   
   \author{Nicholas Cazet}
 
   \begin{abstract}
   
   The axioms of a quandle imply that the columns of its Cayley table are permutations. This paper studies quandles with exactly one non-trivially permuted column. Their automorphism groups, quandle polynomials, (symmetric) cohomology groups, and $Hom$ quandles are studied. The quiver and cocycle invariant of links using these quandles are shown to relate to linking number. 
   \end{abstract}

\maketitle

\section{Introduction}

There are 3 quandle of order 3, the trivial $T_3$, the dihedral $R_3$, and $P_3$.  The quandle $P_3$ fits into a family of quandles $P_n^\sigma$ defined in this paper having only one non-trivially permuted column in their Cayley tables. Several second and third homology groups of $P_n^\sigma$ are calculated since quandle cocycles define link invariants. In addition,  $Hom$ quandles between specific $P_n^\sigma$, the automorphism groups of $P_n^\sigma$, and the quandle polynomials of $P_n^\sigma$  are calculated.

Section \ref{sec:def} defines quandles, their algebraic invariants, and link invariants using quandles. Section \ref{sec:pnsigma} defines $P_n^\sigma$ and studies their algebraic invariants. Section \ref{sec:cohomology} computes several cohomology groups of $P_n^\sigma$. Section \ref{sec:quiver} analyzes the quandle quiver and cocycle invariant of links using $P_n^\sigma$.

\section{Definitions}
\label{sec:def}

\begin{definition}
A {\it quandle}  is a set $X$ with a binary operation $(x,y)\mapsto x*y$ such that 

\begin{enumerate}

\item[(i)] for any $x\in X$,  $x*x=x$,
\item[(ii)] for any $x,y\in X$, there exists a unique $z\in X$ such that $z*y=x$, and 
\item[(iii)] for any $x,y,z\in X$,  $(x*y)*z=(x*z)*{(y*z)}$.

\end{enumerate}

\label{def:quandle}
\end{definition}

\noindent For any $x,y\in X$, let $x \ \bar*\ y$ be the unique $z\in X$ such that $z*y=x$.

\begin{definition}[\cite{symquandle}]

An involution $\rho$ on a quandle $X$ is {\it good} if $\rho(x*y)=\rho(x)*y$ and $x*\rho(y)=x\ \bar* \ y$ for all $x,y\in X$. Such a pair $(X,\rho)$ is a symmetric quandle.

\end{definition}

\begin{example}[\cite{symquandle}]

Conjugation on a group $g*h=h^{-1}gh$ defines a quandle. The inverse map $\rho(g)=g^{-1}$ is a good involution.

\end{example}

\begin{example}
The trivial quandle $T_{n}=\{0,1,\dots, n-1\}$ of order $n$ is defined by $x*y=x$ for all $x,y\in T_{n}$. Any involution on $T_{n}$ is good.
\end{example}

\begin{example}

 The operation $x*y \equiv 2y-x \mod n$ on $\mathbb{Z}_n$ gives the dihedral quandle $R_n$. Kamada and Oshiro classified all good involutions on $R_n$ in \cite{sym}.

\end{example}

Finite quandles have Cayley tables. The table's entries $(\alpha_{ij})$ are given by $\alpha_{ij}=i*j$.  The axioms of a quandle imply that $\alpha_{ii}=i$ and that the columns  of the table represent permutations.  Table \ref{tab:3} gives the Cayley tables of the order 3 quandles.

\begin{table}[ht]
\centering
\begin{tabular}{c|c c c   }

& 0&1&2  \\\hline 
0& 0 & 0 &0\\ 
1&1 & 1&1\\
2&2 & 2 &2\\ 

\end{tabular}
\hspace{4mm}
\begin{tabular}{c|c c c   }

& 0&1&2  \\\hline 
 0&0 & 2 &1\\ 
1 &2& 1 &0\\
2& 1& 0 &2\\

\end{tabular}
\hspace{4mm}
\begin{tabular}{c|c c c   }

& 0&1&2  \\\hline 
0& 0 & 0 &0\\ 
1 &2& 1&1\\
2 &1& 2 &2\\ 

\end{tabular}
\vspace{4mm}
\caption{Cayley tables of $T_3$, $R_3$, and $P_3$.}
\label{tab:3}
\end{table}

\subsection{Automorphism Groups} A {\it quandle homomorphism} satisfies $f(x*y)=f(x)*f(y)$. Under the operation of composition, the bijective endomorphisms on a quandle $X$ constitute its automorphism group $Aut(X)$. For every $y\in X$, let $S_y(x)=x*y$. The second axiom of a quandle implies that $S_y$ is a bijection for any $y\in X$. The group generated by the permutations $S_y$ is the quandle's inner automorphism group $Inn(X)$. Recent articles studying these groups include \cites{bardakov2017automorphism,bardakov2019automorphism,hou2011automorphism,elhamdadi2012automorphism}.

\subsection{Quandle Polynomials}

Nelson introduced a two-variable polynomial invariant of finite quandles in \cite{nelson2008polynomial}.

\begin{definition}

Let $X$ be a finite quandle. For any element $x\in X$, let \[ c(x)=|\{y\in X \ : \ y*x=y\}|, \ \text{and} \ r(x)=|\{y\in X: x*y=x\}|.\] The {\it quandle polynomial of $X$} is $qp_X(s,t)=\sum_{x\in X} s^{r(x)}t^{c(x)}.$ 
\end{definition}

\noindent He later generalized the polynomial in \cite{nelson2011generalized}.

\subsection{$Hom$ Quandles}

Crans and Nelson showed that when $A$ is an abelian quandle the set of quandle homomorphism $Hom(X,A)$ from a quandle $X$ to $A$ has a natural quandle structure \cite{crans2014hom}.  

\begin{definition}

A quandle $A$ is abelian or medial if for all $x,y,z,w\in Q$, \[ (x*y)*(z*w)=(x*z)*(y*w).\]

\end{definition}

\begin{theorem}[Crans, Nelson '14 \cite{crans2014hom}]

Let $X$ and $A$ be quandles. If $A$ is abelian, then $Hom(X,A)$ is an abelian quandle under the pointwise operation $(f*g)(x)=f(x)*g(x)$.  Let $X$ be a finitely generated quandle and $A$ an abelian quandle. Then $Hom(X,A)$ is isomorphic to a subquandle of $A^c$ where $c$ is the minimal number of generators of $X$. Identify an element $f\in Hom(X,A)$ with the $c$-tuple  \[ (f(x_1),\dots,f(x_c)).\] The pointwise operation on $Hom(X,A)$ agrees with the componentwise operation on the product $A^c$.
\end{theorem}

\noindent For a finite quandle $X=\{0,1,\dots,n-1\}$, it is sufficient to classify the componentwise quandle structure of the $n$-tuples $(f(0),f(1),\dots,f(n-1))$ to determine the quandle structure of $Hom(X,A)$.  $Hom$ and abelian quandles were recently studied in \cites{bonatto2019structure,jedlivcka2015structure,jedlivcka2018subdirectly,lebed2021abelian}.

\subsection{Quandle Coloring Number} Let $X$ be a finite quandle. Assigning $X$ elements to the arcs of an oriented link diagram so that each crossing satisfies the quandle coloring condition of Figure \ref{fig:colorlink} is called a {\it quandle coloring}, specifically an $X${\it-coloring}. The axioms of a quandle imply that an $X$-coloring uniquely extends to another after any of the 3 Reidemeister moves, see \cites{elhamdadi2015quandles,CKS,kamada2017surface}. For a finite quandle $X$, the {\it quandle coloring number} $col_X(L)$ is a link invariant and equal to the number of $X$-colorings of any diagram of $L$.

\begin{figure}[ht]

\begin{overpic}[unit=.434mm,scale=.8]{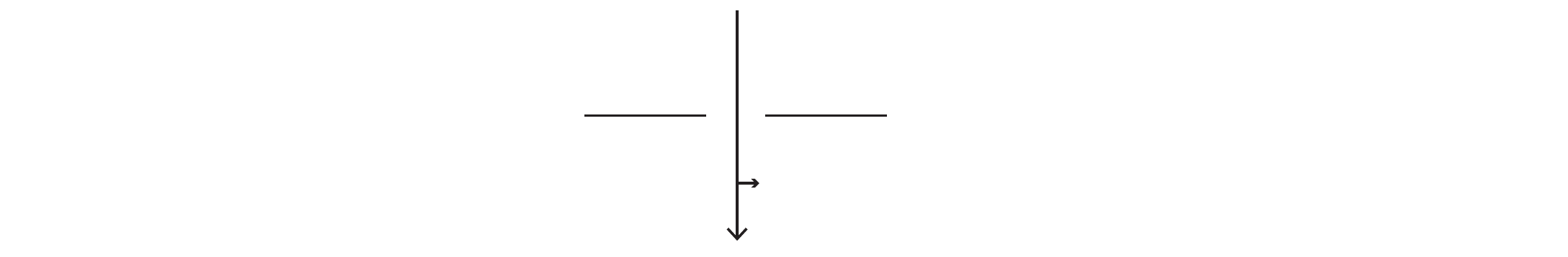}\put(161,39){$x$}
\put(180,50){$y$}\put(202,39){$x*y$}
\end{overpic}

\caption{Quandle coloring condition.}\label{fig:colorlink}
\end{figure}

\subsection{Quandle Quivers}

 A quandle endomorphism $f\in End(X)$ acts on an $X$-coloring of a link diagram by replacing the arcs' colors by their images under $f$. Cho and Nelson introduced the {\it quandle quiver invariant} of links enhancing and categorifying the coloring number \cite{cho2019quandle}.

\begin{definition}

Let $X$ be a finite quandle and $S\subset End(X)$. The quandle quiver $\mathcal{Q}_{X}^S(L)$ of a link $L$ is a directed graph with vertices representing the $X$-colorings of $L$ and directed edges representing the action of $S$ on its colorings. 

\end{definition}

\noindent Two quivers are isomorphic if there exists a bijection $f$ between their vertices that extends to a bijection $F$ between their edges, $F(v_1v_2)=f(v_1)f(v_2)$ for every edge $v_1v_2$.

  \begin{theorem}
  
  The quandle quiver $\mathcal{Q}_{T_n}^S(L)$ is determined by the number of components in $L$.
  \end{theorem}

 \begin{proof}  Let $L$ be a link with $k$ components. A $T_n$-coloring of $L$ is $k$ independent choices of colors, one for each component.  Let $\varphi:\pi_0(L)\to\pi_0(L')$ be a bijection between the components of $L$ and the components of another link $L'$ with $k$ components.   The bijection $\varphi$ extends to a bijection $\Phi:\text{Col}_{T_n}(L)\to \text{Col}_{T_n}(L')$.  Let $f\in$ End$(T_n)$. Sending the edge connecting the colorings $c$ and $f\circ c$ to the edge connecting $\Phi(c)$ and $\Phi(f\circ c)$ defines a quiver isomorphism.
 \end{proof}

 \begin{theorem}[Taniguchi '20 \cite{taniguchi2021quandle}]
 
 Let $L$ and $L'$ be oriented link diagrams and $p$ be a prime. For any $S\subseteq End(R_p)$, the quandle quivers $Q_{R_p}^S(L)$ and  $Q_{R_p}^S(L')$ are isomorphic if and only if $col_{R_p}(L)=col_{R_p}(L')$.
 \label{thm:taniguchi}
 
 \end{theorem}

\noindent Recent studies of quandle quivers include \cites{kim2021quandle,zhou2023quandle,p2}.

\subsection{Quandle Cohomology}

Quandle cohomology was introduced by Carter, Jelsovsky, Kamada, Langford, and Saito to study classical oriented links and oriented surface-links \cite{cocycle}. Symmetric quandle cohomology was later introduced by Kamada to study classical unoriented links and unoriented surface-links \cite{symquandle}.

Let $X$ be a quandle. Let $C_n(X)$ be the free abelian group generated by $X^n$ when $n\geq 1$ and $C_n(X)=0$ when $n<1$. Define the boundary homomorphism $\partial_n : C_n(X)\to C_{n-1}(X)$ by \[ \partial_n(x_1,\dots,x_n)=\sum_{i=1}^n(-1)^i \{(x_1,\dots,\hat x_i,\dots, x_n)-(x_1*x_i,\dots, x_{i-1}*x_i,\hat x_i,x_{i+1},\dots,x_n)\} \]

\noindent for $n>1$ and $\partial_n=0$ for $n\leq 1$. Then $C_*(X)=\{C_n(X),\partial_n\}$ is a chain complex \cites{cocycle,CKS,Carter1999ComputationsOQ,kamada2017surface}.  
For $n\geq 2$, let $D_n^Q(X)$ be the subgroup of $C_n(X)$ generated by the elements \[ \bigcup_{i=1}^{n-1}\{(x_1,\dots,x_n)\in X^n \ | \ x_i=x_{i+1}\},\] 

\noindent and for $n<2$, let $D_n^Q(X)=0$. For a good involution $\rho:X\to X$, let $D_n^\rho(X)$ be the subgroup of $C_n(X)$ generated by the elements 
\[ \bigcup_{i=1}^{n}\left\{(x_1,\dots,x_n)+(x_1*x_i,\dots,x_{i-1}*x_i,\rho(x_i),x_{i+1},\dots,x_n) \ | \ x_1,\dots,x_n \in X^n \right\},\] 

\noindent for $n\geq 1$ and $D_n^\rho(X)=0$ for $n<1$.  Then $D_*^Q(X)=\{ D_n^Q(X), \partial_n\}$ and $D_*^\rho(X)=\{ D_n^\rho(X), \partial_n\}$ are subcomplexes of $C_*(X)$ \cite{sym}. Let $C_*^Q(X)=C_*(X)/D_*^Q(X)$ and $C_*^{Q,\rho}(X)=C_*(X)/(D_*^Q(X)+D_*^\rho(X))$

For an abelian group $A$, the cohomology groups of the complexes \[ C_*^Q(X;A)=C_*^Q(X)\otimes A \ \text{and} \  C_*^{Q,\rho}(X;A)=C_*^{Q,\rho} \otimes A\] are the {\it quandle cohomology groups} $H_Q^*(X;A)$ and the {\it symmetric quandle cohomology groups} $H_{Q,\rho}^*(X;A)$ of $X$.

\begin{lemma}[\cites{cocycle,CKS,Carter1999ComputationsOQ,kamada2017surface}] Let $X$ be a quandle and $A$ an abelian group. A homomorphism $f: C_2(X) \to A$ is a {\it quandle 2-cocycle} of $X$ if the following conditions are satisfied: 

\begin{enumerate}
\item[(i)] For any $x_0,x_1,x_2\in X,$ \begin{align*}
&f(x_0,x_1)+f(x_0*{x_1},x_2) \\ 
&-f(x_0,x_2)-f(x_0*{x_2},x_1*{x_2})=0
\end{align*}

\item[(ii)] For all $x\in X$, $f(x,x)=0$.

\end{enumerate}
% A homomorphism $f: C_3(X) \to A$ is a {\it quandle 3-cocycle} of $X$ if the following conditions are satisfied: 
%
%\begin{enumerate}
%\item[(i)] For any $x_0,x_1,x_2,x_3\in X,$ \begin{align*}
%&-f(x_0,x_2,x_3)+f(x_0*{x_1},x_2,x_3)+f(x_0,x_1,x_3) \\ 
%&-f(x_0*{x_2},x_1*{x_2},x_3)-f(x_0,x_1,x_2)+f(x_0*{x_3},x_1*{x_3},x_2*{x_3})=0
%\end{align*}
%
%\item[(ii)] For any $x,y\in X$, $f(y,x,x)=f(x,x,y)=0$.
%
%\end{enumerate}

\label{lem:cocycle}
\end{lemma}

\subsection{Quandle Cocycle Invariant}

In their seminal article defining quandle cohomology \cite{cocycle}, Carter et al. defined the quandle cocycle invariant and used it to show that the 2-twist-spun trefoil is not isotopic to its orientation reverse. Notable use of the invariant also includes \cites{satoh1,satohcocycle,hat,satoh6,kamadakim,ribbon}. Using a quandle 2-cocycle, this invariant can be defined for classical links. This paper will only use the quandle cocycle invariant formulated for classical links, i.e. using a quandle 2-cocycle.

For a finite quandle $X$, let $\theta$ be a quandle 2-cocycle with coefficients in an abelian group $A$ written multiplicatively. Let $L$ be an $X$-colored link diagram and $C$ denote the set of its crossings. For $c\in C$, suppose that $x$ is the color assigned to the under arc for which the orientation normal of the over arc points away from and $y$ is the color of the over arc as in Figure \ref{fig:colorlink}. Define $\theta_c=\theta(x,y)^\epsilon$ to be the $\theta$-weight of the $X$-colored crossing $c$ where $\epsilon$ is $+1$ if the crossing is positive and $-1$ if the crossing is negative. The $\theta$-weight of an $X$-colored link diagram is the product of all  $\theta_c$'s. The quandle cocycle invariant is the formal sum of all $\theta$-weights. The {\it quandle cocycle invariant} of $L$ is \[ \Phi_\theta(L)=\sum_{\text{$X$-colorings}}\ \prod_{c\in C}\theta_{c}\in \mathbb{Z}[A].\] A proof that $\Phi_\theta(L)$ does not depend on the diagram of $L$ can be found in \cites{cocycle,kamada2017surface,CKS}.

\section{The quandle $P_n^\sigma$}
\label{sec:pnsigma}

   %%%%%%%%%%%%%%%%%%%%%%%%%%%%%%%%%%%%%%%%%%

  Let $S_n$ be the symmetric group on the first $n$ positive integers and $\sigma \in S_n$ a permutation. Define the operation $*$ on $P_{n}^\sigma =\{0,1,2,\dots,n\}$ by 
  \[ x*y=\begin{cases} 
      x &  y\neq0 \\
    \sigma(x) & y=0, x\neq 0  \\
      0 & y=0,x=0.
   \end{cases}
\] This is also defined by permuting the positive entries in the first column of the trivial quandle $T_{n+1}$'s Cayley table with respect to $\sigma$, see Table \ref{tab:P}.

\begin{table}[ht]

\begin{tabular}{c|c c c c c  }

& 0&1&2 & $\cdots$ & $n$ \\\hline 
 0 & 0 &0&0 & $\cdots$ & 0\\ 
1 & $\sigma(1)$ &1&1 & $\cdots$ & 1\\
2 & $\sigma(2)$ &2&2& $ \cdots $& 2\\ 
 $\vdots $&$\vdots $  &$\vdots $&$\vdots $& $\ddots $& $\vdots $\\ 
$n$ & $\sigma(n)$ &$n$&$n$&  $\cdots$ &$ n$\\ 
\end{tabular}
\vspace{4mm}
\caption{Cayley table of $P_{n}^\sigma$.}
\label{tab:P}
\end{table}
  
  \begin{proposition}
  
$(P_{n}^\sigma, *)$ is a quandle.
  \label{prop:quandle}
  \end{proposition}

  \begin{proof}
  
The diagonal of its Cayley table shows that the first axiom is satisfied and since each column is a permutation the second axiom satisfied. The third axiom \[ (x*y)*z=(x*z)*(x*z)\] follows from a case study, see Table \ref{tab:quandle}.   
    \end{proof}
  
  \begin{table}[ht]
\centering
\begin{tabular}{|c|c| c| c| }
\hline
$x$ & $y$&$z$ & $(x*y)*z=(x*z)*(y*z)$ \\\hline \hline 
+ &+& + & $x=x$ \\\hline 
+ &+& 0 & $\sigma(x)=\sigma(x)$ \\\hline 
+ &0& + & $\sigma(x)=\sigma(x)$ \\\hline 
0 &+& + & $0=0$ \\\hline 
+ &0& 0 & $\sigma^2(x)=\sigma^2(x)$ \\\hline 
0 &+& 0 & $0=0$ \\\hline 
0 &0& + & $0=0$ \\\hline 
0 &0& 0 & $0=0$ \\\hline 
\end{tabular}
\vspace{4mm}
\caption{The cases of Proposition \ref{prop:quandle}.}
\label{tab:quandle}
\end{table}

  If $\sigma=id$, then $P_{n}^\sigma$ is the trivial quandle $T_{n+1}$.

  \begin{lemma}
  
 If $\sigma\neq id$, then any quandle isomorphism between $P_n^\sigma$ and  $P_n^\tau$ fixes 0.  
 
 \label{lem:fix0}
  \end{lemma}
  
  \begin{proof}
  
  Suppose that $f$ is an isomorphism with $f(0)\neq0$. For any $a\neq 0$, $f(\sigma(a))=f(a*0)=f(a)*f(0)=f(a)$ implying that $\sigma(a)=a$.

  \end{proof}

  \begin{theorem}

$  P_n^\sigma \cong P_n^\tau$ if and only if $\sigma \ \text{is conjugate to} \ \tau.$
  \label{thm:conjugate}
  
  \end{theorem}

  \begin{proof}
  
  Suppose that $f: P_n^\sigma\to P_n^\tau$ is a quandle isomorphism and $x\in P_n$ is positive. Lemma \ref{lem:fix0} implies that $f(0)=0$. Therefore, $f(\sigma(x))=f(x*0)=f(x)*f(0)=f(x)*0=\tau(f(x))$, i.e.  $\sigma=f^{-1}\tau f$ on the positive integers. 

Now suppose that $\sigma=h^{-1}\tau h$ for some permutation $h\in S_n$. Extend $h$ to $h':P_n^\sigma\to P_n^\tau$ by setting $h'(0)=0$ and $h'(x)=h(x)$ for $x\neq 0$. A case study shows that $h'$ is a quandle homomorphism thus isomorphism. If $a,b\neq 0$, then $h'(a*b)=h'(a)=h(a)=h(a)*h(b)=h'(a)*h'(b)$ since $h(b)\neq 0$. If $a\neq b=0$, then $h'(a*b)=h(\sigma(a))=\tau(h(a))=h(a)*0=h'(a)*h'(b).$ If $a=0$, then  $h'(a*b)=h'(0)=0=0*h'(b)=h'(a)*h'(b)$ for any $b$.

  \end{proof}

    \begin{theorem}
  
If $\sigma\neq id$, then  $Aut(P_n^\sigma)\cong C_{S_n}(\sigma)$ and $Inn(P_n^\sigma)\cong \mathbb{Z}_{|\sigma|}$ where $C_{S_n}(\sigma)$ denotes $\sigma$'s centralizer and $|\sigma|$ is the order of $\sigma$.
  
  \end{theorem}

\begin{proof}

For any $f\in Aut(P_n^\sigma)$, let $f_+$ be its restriction to $P_n^\sigma\setminus\{0\}=S_n$. Lemma \ref{lem:fix0} implies that any two automorphisms $f$ and $g$ satisfy $(g\circ f)_+=g_+\circ f_+$. For any positive $x\in P_n^\sigma$, 
\[f_+(\sigma(x))=f_+(x*0)=f(x*0)=f(x)*f(0)=f_+(x)*0=\sigma(f_+(x)),\] so $f_+\in C_{S_n}(\sigma).$ Therefore, the restriction map $f\mapsto f_+$ is a group homomorphism from $Aut(P_n^\sigma)$ to $C_{S_n}(\sigma)$. 

 For any $g \in C_{S_n}(\sigma)$, let $g_0:P_n^\sigma\to P_n^\sigma$ be its extension where $g_0(x)=g(x)$ for $x\neq 0$ and $g_0(0)=0$. A case study shows that the bijection $g_0$ is an automorphism. If $a,b\neq 0$, then $g_0(a*b)=g(a)=g(a)*g(b)=g_0(a)*g_0(b)$. If $a\neq b=0$, then \[g_0(a*b)=g_0(\sigma(a))=g(\sigma(a))=\sigma(g(a))=\sigma(g_0(a))=g_0(a)*0=g_0(a)*g_0(0).\]

\noindent If $b\neq a=0$ or $a=b=0$, then $g_0(a*b)=0=g_0(a)*g_0(b)$. Altogether, $g_0$ is an automorphism and the extension map $g\mapsto g_0$ is inverse to the restriction map $f\mapsto f_+$.

There are two permutations that generate $Inn(P_n^\sigma),$ $S_0(x)=x*0=\sigma(x)$ and $S_{a\neq0}(x)=x*a=x$. These generate the cyclic group of the same order as $\sigma$.

\end{proof}

\begin{theorem}

Let $\alpha$ be the number of positive integers fixed by a non-trivial $\sigma$. The quandle polynomial of $P_n^\sigma$ is \[ qp_{P_n^\sigma}(s,t)=\alpha s^{n+1}t^{n+1}+(n-\alpha)s^nt^{n+1}+s^{n+1}t^{1+\alpha}.\]

\end{theorem}

\begin{proof}
For a positive $x$ fixed by $\sigma$, $r(x)=c(x)=n+1$. If $x$ is positive and not fixed by $\sigma$, then $r(x)=n$ and $c(x)=n+1$. Finally, $r(0)=n+1$ and $c(0)=1+\alpha$.
\end{proof}

\begin{remark}
The quandle polynomial distinguishes the isomorphism classes of $P_n^\sigma$ and $P_n^\tau$ if and only if $\sigma$ and $\tau$ have a different number of fixed points.
\end{remark}

\begin{lemma}

$P_n^\sigma$ is an abelian for any $\sigma$.
\label{lem:abelian}
\end{lemma}

\begin{proof}

The abelian condition \[ (x*y)*(z*w)=(x*z)*(y*w)\] is verified by checking all permutations of the variables being positive or zero. If $x=0$, the equation simplifies to $0=0$ for any $y,z$, and $w$. Table \ref{tab:abelian} verifies the cases where $x\neq 0$. 

\end{proof}

\begin{table}[ht]
\centering
\begin{tabular}{|c|c| c| c| c|}
\hline
$x$ & $y$&$z$&$w$ & $(x*y)*(z*w)=(x*z)*(y*w)$ \\\hline \hline 
$+$ & +&$+$&$+$ & $x=x$ \\\hline 
$+$ & 0& 0 & 0 & $\sigma^2(x)=\sigma^2(x)$ \\\hline 
$+$ & +& 0 & 0 & $\sigma(x)=\sigma(x)$ \\\hline 
$+$ & 0& + & 0 & $\sigma(x)=\sigma(x)$ \\\hline 
$+$ & 0& 0 & + & $\sigma^2(x)=\sigma^2(x)$ \\\hline 
$+$ & +& + & 0 & $x=x$ \\\hline 
$+$ & +& 0 & +& $\sigma(x)=\sigma(x)$ \\\hline 
$+$ & 0& + & + & $\sigma(x)=\sigma(x)$ \\\hline 
\end{tabular}
\vspace{4mm}
\caption{The cases of Lemma \ref{lem:abelian} where $x\neq 0$.}
\label{tab:abelian}
\end{table}

\begin{theorem}
Let $\sigma$ and $\sigma'$ be $n$-cycles. Table \ref{tab:cycle} gives the Cayley table of  \[ Hom(P_n^{\sigma},P_n^{\sigma'}).\] 
\end{theorem}

\begin{proof}

Assume that $\sigma=\sigma'=(123\cdots n)$ by Theorem \ref{thm:conjugate}. Let $f\in Hom(P_n^{\sigma},P_n^{\sigma'})$. If $f(k)=0$ for some $k\neq 0$,  then  \[f(\sigma(k))=f(k*0)=f(k)*f(0)=0.\]  This implies that $f$ is zero on all positive integers since $\sigma$ is an $n$-cycle. If further $f(0)\neq 0$,  then $f(0)=f(0*k)=f(0)*f(k)=f(0)*0=\sigma'(f(0))$. Since $\sigma'$ has no fixed pointed, $f(0)=0$. Therefore, $f(k)=0$ for some $k\neq 0$ only if $f$ is the zero map.

Now suppose that $f(k)$ is positive whenever $k$ is positive. If $f(0)=0$, then $f$ is determined by where it sends 1. For any positive $k$, $f(k)=f(\sigma^k(1))=f(1*^k0)=f(1)*^k0=(\sigma')^k(f(1)).$ If $f(0)\neq 0$, then $f(1)=f(2)=\cdots=f(n)$ since $f(\sigma(k))=f(k*0)=f(k)*f(0)=f(k)$. 

The map $f$ is determined by its values on $0$ and $1$. If $f$ is not the zero map, then $f(0)\in P_n^{\sigma'}$ and $f(1)\in P_n^{\sigma'}\setminus \{0\}$. The set $Hom(P_n^{\sigma},P_n^{\sigma'})$ is parameterized by the tuples $(f(0),f(1))$ and its quandles structure is described in Table \ref{tab:cycle}.

\begin{table}
\centering
\begin{tabular}{c|c c c c c c c c c c c c c c c}

 & (0,0)&(0,1)&(0,2) & $\cdots$ & $(0,n)$ & (1,1) & $\cdots$ & $(1,n)$& $(2,1)$& $\cdots$ & $(2,n)$ & $\cdots$ & $(n,1)$&$\cdots$ & $(n,n)$ \\ \hline 
 (0,0) & (0,0) & (0,0)& (0,0) & $\cdots$  &(0,0) & (0,0)& $\cdots$ & (0,0)& (0,0)& $\cdots$ & (0,0)& $\cdots$ & (0,0)& $\cdots$ & (0,0)\\ 
(0,1) &(0,2) & (0,1) & (0,1)&$\cdots$ &(0,1) &(0,1)&$\cdots$&(0,1)&(0,1)& $\cdots$ & (0,1)& $\cdots$ & (0,1)& $\cdots$& (0,1) \\
(0,2) &(0,3) & (0,2) & (0,2)&$\cdots$ &(0,2) &(0,2)&$\cdots$&(0,2)&(0,2)& $\cdots$ & (0,2)& $\cdots$ & (0,2)& $\cdots$& (0,2) \\
$\vdots$ &$\vdots$  &$\vdots$&$\vdots$& $\ddots$& $\vdots$&$\vdots$&$\ddots$&$\vdots$&$\vdots$& $\ddots$ & $\vdots$ &$\ddots$ & $\vdots$& $\ddots$& $\vdots$ \\ 
$(0,n)$ &(0,1) & $(0,n)$ & $(0,n)$&$\cdots$ &$(0,n)$ &$(0,n)$&$\cdots$&$(0,n)$&$(0,n)$& $\cdots$ & $(0,n)$& $\cdots$ & $(0,n)$& $\cdots$& $(0,n)$ \\
(1,1) & (2,2) &(2,1)&(2,1)&  $\cdots$ &(2,1)&(1,1)& $\cdots$ & (1,1)& (1,1)& $\cdots$ & (1,1)& $\cdots$ & (1,1)& $\cdots$ & (1,1)\\ 
(1,2) & (2,3) &(2,2)&(2,2)&  $\cdots$ &(2,2)&(1,2)& $\cdots$ & (1,2)& (1,2)& $\cdots$ & (1,2)& $\cdots$ & (1,2)& $\cdots$ & (1,2)\\ 
$\vdots$ &$\vdots$  &$\vdots$&$\vdots$& $\ddots$& $\vdots$&$\vdots$&$\ddots$&$\vdots$&$\vdots$& $\ddots$ & $\vdots$ &$\ddots$ & $\vdots$& $\ddots$& $\vdots$ \\ 
$(1,n)$& (2,1) &  $(2,n)$ &$(2,n)$ &$\cdots$&$(2,n)$&$(1,n)$&$\cdots$& $(1,n)$& $(1,n)$& $\cdots$ & $(1,n)$& $\cdots$ & $(1,n)$& $\cdots$ & $(1,n)$\\ 
$(2,1)$& (3,2) &  (3,1) &(3,1) &$\cdots$&(3,1)&$(2,1)$&$\cdots$& $(2,1)$& $(2,1)$& $\cdots$ & $(2,1)$& $\cdots$ & $(2,1)$& $\cdots$ & $(2,1)$\\ 
$(2,2)$& (3,3) &  (3,2) &(3,2) &$\cdots$ &(3,2) &$(2,2)$&$\cdots$& $(2,2)$& $(2,2)$& $\cdots$ & $(2,2)$& $\cdots$ & $(2,2)$& $\cdots$ & $(2,2)$\\ $\vdots$ &$\vdots$  &$\vdots$&$\vdots$& $\ddots$& $\vdots$&$\vdots$&$\ddots$&$\vdots$&$\vdots$& $\ddots$ & $\vdots$ &$\ddots$ & $\vdots$& $\ddots$& $\vdots$ \\ 
$(2,n)$& (3,1) &  $(3,n)$ &$(3,n)$ &$\cdots$&$(3,n)$&$(2,n)$&$\cdots$& $(2,n)$& $(2,n)$& $\cdots$ & $(2,n)$& $\cdots$ & $(2,n)$& $\cdots$ & $(2,n)$\\ 
$\vdots$ &$\vdots$  &$\vdots$&$\vdots$& $\ddots$& $\vdots$&$\vdots$&$\ddots$&$\vdots$&$\vdots$& $\ddots$ & $\vdots$ &$\ddots$ & $\vdots$& $\ddots$& $\vdots$ \\ 
$(n,1)$& (1,2) &  (1,1) &(1,1) &$\cdots$&(1,1)&$(n,1)$&$\cdots$& $(n,1)$& $(n,1)$& $\cdots$ & $(n,1)$& $\cdots$ & $(n,1)$& $\cdots$ & $(n,1)$\\ 
$(n,2)$& (1,3) &  (1,2) &(1,2) &$\cdots$&(1,2)&$(n,2)$&$\cdots$& $(n,2)$& $(n,2)$& $\cdots$ & $(n,2)$& $\cdots$ & $(n,2)$& $\cdots$ & $(n,2)$\\ $\vdots$ &$\vdots$  &$\vdots$&$\vdots$& $\ddots$& $\vdots$&$\vdots$&$\ddots$&$\vdots$&$\vdots$& $\ddots$ & $\vdots$ &$\ddots$ & $\vdots$& $\ddots$& $\vdots$ \\ 
$(n,n)$& (1,1) &  $(1,n)$ &$(1,n)$ &$\cdots$&$(1,n)$&$(n,n)$&$\cdots$& $(n,n)$& $(n,n)$& $\cdots$ & $(n,n)$& $\cdots$ & $(n,n)$& $\cdots$ & $(n,n)$\\ 
\end{tabular}
\vspace{4mm}
\caption{Cayley table of $Hom(P_n^{n-cycle},P_n^{n-cycle})$.}
\label{tab:cycle}
\end{table}

\end{proof}

\begin{corollary}
Table \ref{tab:hom} gives the Cayley table of $Hom(P_3,P_3)$.
\end{corollary}

\begin{proof}

There are 7 endomorphisms described as the triples $(f(0),f(1),f(2))$: $0=(0,0,0), 1=(0,1,2), 2=(0, 2, 1), 3=(1,1,1), 4=(2,2,2),  5=(1,2,2), 6=(2,1,1)$. 

\begin{table}[ht]
\centering
\begin{tabular}{c|c c c c c c c }

 & 0&1&2 & 3& 4&5&6 \\\hline 
 0 & 0 &0&0 & 0 & 0&0&0\\ 
1 &2 & 1 & 1&1&1 &1&1\\
2 & 1 &2& 2&2&2&2&2\\ 
3&4  &6&6& 3& 3&3&3\\ 
4 & 3 &5&5&  4 &4&4&4\\ 
5 & 6 &  4 &4&5&5&5&5\\ 
6& 5 &  3 &3 &6&6&6&6\\ 
\end{tabular}
\vspace{4mm}
\caption{Cayley table of $Hom(P_3,P_3) $.}
\label{tab:hom}
\end{table}

\end{proof}

\section{Cohomology groups of $P_n^\sigma$}
\label{sec:cohomology}

 The goal of this section is to compute $H^2_Q(P_n^\sigma;A)$, and classify the good involutions of $P_n^\sigma$ to compute $H^2_{Q,\rho}(P_3;\mathbb{Z}_2)$.  The computations will follow the notation of \cite{cocycle} where the coefficient group $A$ is taken to be $\mathbb{Z}$, $\mathbb{Z}_p$, or $\mathbb{Q}$.

\begin{theorem}
Let $k$ be the number of cycles in $\sigma$. For any $A$, \[H^2_Q(P_n^\sigma;A)=A^{k^2+k}.\]
\label{thm:cohom2}
\end{theorem}

\begin{proof}

A 2-cocycle $\phi\in Z^2(P_n^\sigma;A)$ can be expressed as $\phi=\sum_{i,j\in P_n^\sigma} C_{(i,j)} \rchi_{(i,j)}$ such that

\[ C_{(p,r)}+C_{(p*r,q*r)}=C_{(p,q)}+C_{(p*q,r)},\]

\noindent and $C_{(p,p)}=0$ for all $p,q,r\in P_n^\sigma$, see Lemma \ref{lem:cocycle}.  Checking all permutations of $p,q,$ and $r$ being 0 determines relations among the coefficients, see Table \ref{tab:homology2}.

\begin{table}[ht]
\centering
\begin{tabular}{|c|c| c| c| c|}
\hline
$p$ & $q$ & $r$ & $C_{(p,r)}+C_{(p*r,q*r)}=C_{(p,q)}+C_{(p*q,r)}$ \\\hline \hline 
0 & 0 & 0 & $0=0$ \\\hline 
+ & 0 & 0 & $C_{(p,0)}+C_{(\sigma(p),0)}=C_{(\sigma(p),0)}+C_{(p,0)}$ \\\hline 
0& + & 0 & $C_{(0,q)}=C_{(0,\sigma(q))}$ \\\hline 
0& 0 & + & $C_{(0,r)}=C_{(0,r)}$ \\\hline 
+& + & 0& $C_{(\sigma(p),\sigma(q))}=C_{(p,q)}$ \\\hline 
+& 0& +& $C_{(\sigma(p),r)}=C_{(p,r)}$ \\\hline 
0& +& +& $C_{(0,r)}+C_{(0,q)}=C_{(0,q)}+C_{(0,r)}$ \\\hline 
+ & + & + & $C_{(p,r)}+C_{(p,q)}=C_{(p,q)}+C_{(p,r)}$ \\\hline 
\end{tabular}
\vspace{4mm}
\caption{Coefficient relations of Theorem \ref{thm:cohom2}.}
\label{tab:homology2}
\end{table}

This gives the non-trivial relations \[C_{(0,\sigma(q))}=C_{(0,q)},\]  \[C_{(\sigma(p),\sigma(q))}=C_{(p,q),}\]  \[C_{(\sigma(p),r)}=C_{(p,r)},\] for all $p,q,r\neq 0$. Combining the second and third relation gives that $C_{(\sigma^l(a),\sigma^m(b))}=C_{(a,b)}$ for all $l,m$ and  $a,b\neq 0$. Therefore, $C_{(a,b)}=C_{(c,d)}$ whenever $orb(a)=orb(c)$ and $orb(b)=orb(d)$.

Consider the coboundaries of $C^1_Q(P_n^\sigma;A)$'s generators $\rchi_m$. If $m=0,$

\[ \delta \rchi_0(x_0,x_1)=\rchi_0(x_0)-\rchi_0(x_0*x_1)=0.\]

\noindent If $m\neq0$, 
\[ \delta \rchi_m(x_0,x_1)=\rchi_m(x_0)-\rchi_m(x_0*x_1)= 
\rchi_{(m,0)}- \rchi_{(\sigma^{-1}(m),0)}.   
 \]

\noindent Therefore, $\rchi_{(p,0)}$ is cohomologous to $\rchi_{(q,0)}$ whenever $orb_\sigma(p)=orb_\sigma(q)$. Let $m_1$, $m_2$, \dots, $m_k$ represent the cycles of $\sigma$. The relations and coboundary information give that $\phi$ is cohomologous to \begin{align*} \sum_{j=1}^k \left[C_{(0,m_j)}\left(\sum_{i\in orb_\sigma(m_j)} \rchi_{(0,i)}\right)\right]+&\sum_{j=1}^k \left[\left( \sum_{i\in orb_\sigma(m_j)} C_{(i,0)}\right) \rchi_{(m_j,0)}\right]\\ &+\sum_{j,h=1,j\neq h}^k\left[C_{(m_j,m_h)}\left(\sum_{i\in orb_\sigma(m_j)}\sum_{t\in orb_\sigma(m_h)}\rchi_{(i,t)}\right)\right],\end{align*} showing that the second homology group is free on $2k$ plus $k$ permute 2 generators.

\end{proof}

The following theorem characterizes $P_n^\sigma$'s good involutions. This will be used to calculate a second symmetric quandle cohomology group of $P_3$. 
 \begin{theorem}

If $\sigma$ is not an involution, then $P_n^\sigma$ has no good involutions. For a non-trivial involution $\sigma$, an involution $\rho$ on $P_n^\sigma$ is good if and only if $\rho(0)=0$ and as a permutation on the positive integers $\rho\in C_{S_n}(\sigma)$.

\end{theorem}

\begin{proof}

Suppose that $\rho$ is a good involution on $P_n^\sigma$ with $\rho(0)\neq 0$. Then, $x=x*\rho(0)=x \ \bar* \ 0=\sigma^{-1}(x)$ for all $x\neq 0$ implying that $\sigma=id$.  Therefore, $\rho(0)=0$ and $\sigma(x)=x*0=x*\rho(0)=x\ \bar* \ 0=\sigma^{-1}(x)$ for all $x\neq 0$, i.e. $\sigma$ is an involution. Since $\rho(\sigma(x))=\rho(x*0)=\rho(x)*0=\sigma(\rho(x))$ for all $x\neq 0$, $\rho$ is in $\sigma$'s centralizer. 

Now suppose that $\rho$ is an involution on $P_n^\sigma$ fixing 0 and restricting to a permutation in the centralizer of $\sigma$.  A case study shows that $\rho$ is good, see Table \ref{tab:involution}.

\begin{table}[ht]
\centering
\begin{tabular}{|c|c| c| c| c|}
\hline
$x$ & $y$  & $\rho(x*y)=\rho(x)*y$ & $x*\rho(y)=x \ \bar *\  y$  \\\hline \hline 
+ & +  & $\rho(x)=\rho(x)$ & $x=x$  \\\hline  
+ & 0 & $\rho(\sigma(x))=\sigma(\rho(x))$ & $\sigma(x)=\sigma^{-1}(x)$  \\\hline  
0 & +  & $0=0$ & $0=0$  \\\hline  
0 & 0  & $0=0$ & $0=0$  \\\hline  
\end{tabular}
\vspace{4mm}
\caption{Verification that $\rho$ is good.}
\label{tab:involution}
\end{table}

\end{proof}

\begin{theorem}
\label{thm:symgroup}

Let $\rho:P_3\to P_3$ be the transposition $(12)$. Then, \[ H^2_{Q,\rho}(P_3;\mathbb{Z}_2 )=\mathbb{Z}_2.\]

\end{theorem}

\begin{proof}

Theorem \ref{thm:cohom2} shows that any cocycle $\sum_{i,j\in P_n^\sigma} C_{(i,j)}\rchi_{(i,j)} \in Z^2_{Q,\rho}(P_3;\mathbb{Z}_2 )$ is cohomologous to 

\[ C_{(0,1)}(\rchi_{(0,1)}+\rchi_{(0,2)})+(C_{(1,0)}+C_{(2,0)})\rchi_{(1,0)} \]

\noindent in $H^2_{Q,\rho}(P_3;\mathbb{Z}_2 )$. Since $2(0,1), (1,0)+(2,0)\in D_2^\rho$, the coefficients satisfy $2C_{(0,1)}=C_{(1,0)}+C_{(2,0)}=0$. Thus,  \[ \rchi_{(0,1)}+\rchi_{(0,2)} \] represents a generator of the order 2 group $H^2_{Q,\rho}(P_3;\mathbb{Z}_2 )$.

\end{proof}

\section{The quiver and cocycle invariant of links using $P_n^\sigma$}
\label{sec:quiver}

In this section, the cocycle $\theta\in Z_Q^2(P_n^{n-cycle};\langle t^k: k\in \mathbb{Z} \rangle )$ given by \[\theta=t^{\LARGE\mathlarger\sum_{i=1}^n \rchi_{(0,i)}}\]  will be used to compute the cocycle invariant $\Phi_\theta$. This was shown to be a cocycle in Theorem \ref{thm:cohom2}. The cocycle takes the value $t$ on any tuple of the form $(0,a)$ where $a$ is positive. Quandle quivers $\mathcal{Q}_{P_n^\sigma}^S$ will be computed using an arbitrary permutation $\sigma$ and $S\subseteq End(P_n^\sigma)$.

\begin{definition}
Let $K_1$ and $K_2$ be components of an oriented link. Figure \ref{fig:crossing} describes the 4 types of crossings between $K_1$ and $K_2$ where the horizontal arcs belong to one component and the vertical arcs belong to the other. Let $n_1$, $n_2$, $n_3$, and $n_4$ denote the number of crossings of type $(i)$, $(ii)$, $(iii)$, and $(iv)$. The pairwise linking number between $K_1$ and $K_2$ is defined by \[ lk(K_1,K_2)=\frac{n_1+n_2-n_3-n_4}{2}=n_1-n_4=n_2-n_3.\]

\end{definition}

\begin{figure}[ht]

\begin{overpic}[unit=.434mm,scale=.7]{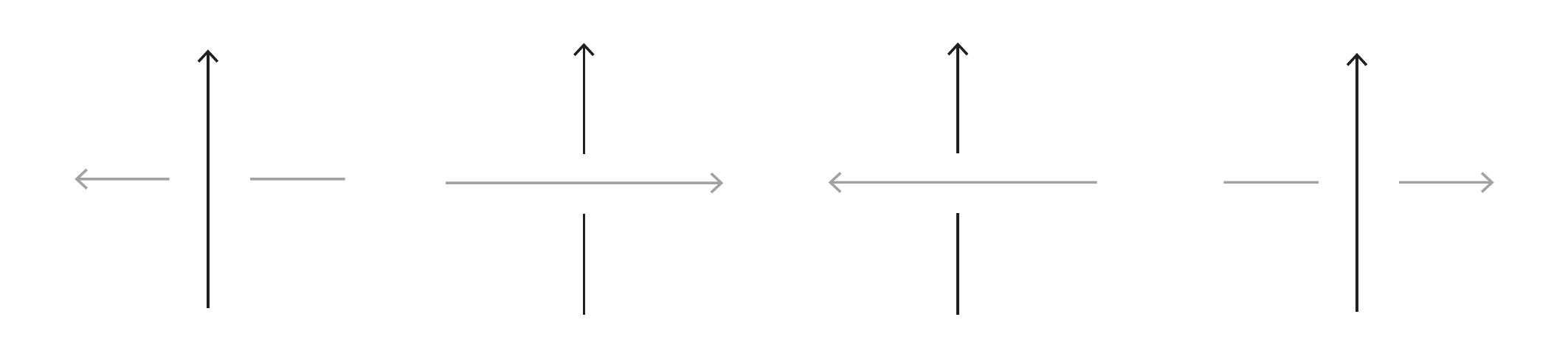}\put(42,-10){$(i)$}\put(124,-10){$(ii)$}\put(205,-10){$(iii)$}\put(295,-10){$(iv)$}
\end{overpic}

\vspace{4mm}
\caption{Crossing types between 2 components.}
\label{fig:crossing}
\end{figure}
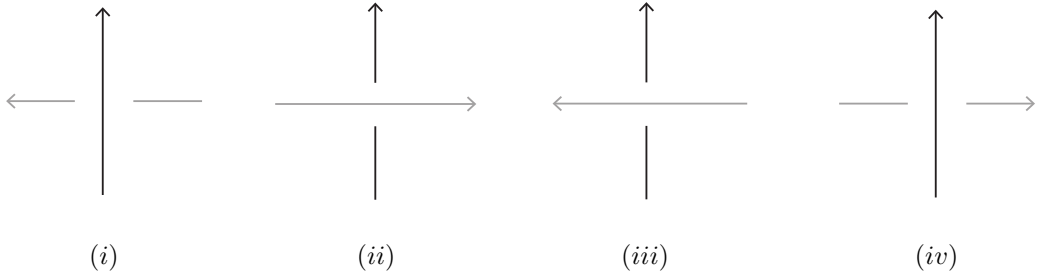

\begin{theorem}

The quiver $\mathcal{Q}_{P_n^\sigma}^S(L)$ of a 2-component link $L=K_1\cup K_2$ is determined by $lk(K_1,K_2)$.

\end{theorem}

\begin{proof}

 Let $a_i$ be some arc of $K_i$. The trivial quandle embeds into $P_n^\sigma$ and coloring $L$ using this subquandle, the positive integers of $P_n^\sigma$, simplifies to independently selecting the colors of $a_1$ and $a_2$. 

If $a_1$ is colored 0 then all arcs of $K_1$ are assigned 0 since $0*x=0$ for any $x\in P_n^\sigma$. If the arc $a_2$ is also colored 0, then the link is trivially colored. If $a_2$ is colored with a positive integer $k$, then $k$'s cycle length $|orb_\sigma(k)|$ must divide the linking number $lk(K_1,K_2)$. To see this, color the arcs of $K_2$ following an orientation and starting with $a_2$. When the arc meets a crossing as an under arc the other under arc must be colored $k$ if the over strand is from $K_2$, $\sigma(k)$ if the over strand is from $K_1$ and the crossing is positive, and $\sigma^{-1}(k)$ if the over strand is from $K_1$ and the crossing is negative. Therefore, $\sigma^{lk(K_1,K_2)}(k)=k$ and $|orb_\sigma(k)|$ divides $lk(K_1,K_2)$.

Let $T$ be the sum of $\sigma$'s cycle lengths which divide $lk(K_1,K_2)$. When $a_1$ is colored 0 there are $T$ choices of colors for $a_2$ that uniquely extend to non-trivial colorings of $L$. Likewise, there are $T$ choices of colors for $a_1$ when $a_2$ is colored 0. The $P_n^\sigma$-colorings of $L$ are parametrized by the colorings $(x_1,x_2)$ of the arcs $a_1$ and $a_2$. The ordered pair $(x_1,x_2)$ represents arc colors of $a_1$ and $a_2$ that uniquely extend to a $P_n^\sigma$-coloring of $L$ if $x_1,x_2\neq0$, $x_1=0$ and $|orb_\sigma(x_2)|$ divides $lk(K_1,K_2)$, $x_2=0$ and $|orb_\sigma(x_1)|$ divides $lk(K_1,K_2)$, or $x_1=x_2=0$. The action of $End(P_n^\sigma)$ on $Col_{P_n^\sigma}(L)$ is determined by its action on the ordered pairs $(x_1,x_2)$ representing colorings. Therefore, the linking number  $lk(K_1,K_2)$ determines the quiver $\mathcal{Q}_{P_n^\sigma}^S(K_1\cup K_2)$.

\end{proof}

  \begin{corollary}

The quandle cocycle invariant $\Phi_\theta$ of a 2-component link $L=K_1\cup K_2$ is determined by $lk(K_1,K_2)$.
\label{cor}
\end{corollary}

\begin{proof}

Suppose that $(x_1,x_2)$ represents a $P_n^\sigma$-coloring of $L$ where $\sigma$ is an $n$-cycle. The $\theta$-weight of $L$ with respect to $(x_1,x_2)$ is 0 when $x_1=x_2=0$ or $x_1,x_2\neq0$. If $x_2\neq x_1=0$, then the corresponding $\theta$-weight is $t^{lk(K_1,K_2)}$. To see this, let $K_1$ represent the horizontal component and let $K_2$ represent the vertical component in Figure \ref{fig:crossing}. Then the crossing types $(i)$, $(ii)$, $(iii)$, and $(iv)$ have  $\theta$-weights $t$, 0, 0, and $t^{-1}$. Since the number of $(i)$ crossings minus the number of $(iv)$ crossings equals the linking number, the corresponding $\theta$-weight of the coloring is $t^{lk(K_1,K_2)}$. Likewise, the $\theta$-weight when $x_1\neq x_2=0$ is $t^{lk(K_1,K_2)}$. This also holds if $K_1$ represented the vertical component and $K_2$ represented the horizontal component in Figure \ref{fig:crossing}.

 If $n$ divides the linking number, then there are $2n$ colorings with a $\theta$-weight of $t^{lk(K_1,K_2)}$, \[(0,1), (0,2), \dots, (0,n), (1,0), (2,0), \dots, (n,0).\] In this case, $\Phi_\theta(L)=1+n^m+2nt^{lk(K_1,K_2)}$. When $n$ does not divide the linking number, $\Phi_\theta(L)=1+n^m$.

\end{proof}

  \begin{definition}
  
  The {\it linking graph} of a link $L$ is a complete graph with vertices representing $L$'s components and edge weights representing the linking numbers between components.    
  \end{definition}
  
\begin{lemma}  Every complete, simple graph with integral weights represents the linking graph of some link. 
\end{lemma}

\begin{proof}

Let $G$ be a complete graph on $n$ vertices with integral weights. Take $n$ points equidistant along the unit circle. The graph $G$ is represented by connecting each pair of points with  a line segment. Assume after a small perturbation that any two arcs intersect transversally in double points along their interiors or not at all. Give these double points arbitrary crossing information. Take small disk neighborhoods of the vertices missing the crossings then remove the interiors of the disks. This leaves a 3-valent, singular link diagram seen as edges connecting unknotted circles,  Figure \ref{fig:singular}. Replace a small neighborhood of the edges with the tangle in Figure \ref{fig:twist} such that the number of twists equals the weight of the edge that it is replacing. Arbitrarily orient the components and replace any twist region with its mirror if the sign of its crossings does not equal the sign of the weight that it replaced. This gives an oriented link whose linking graph is $G$.

\end{proof}

\begin{figure}

\includegraphics[scale=.7]{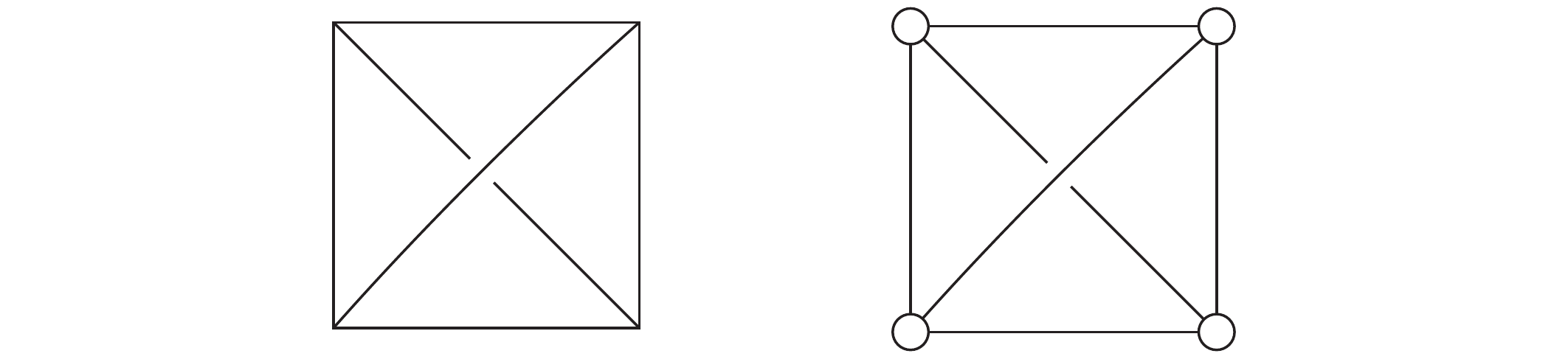}

\caption{Removing the interiors of disk neighborhoods.}
\label{fig:singular}
\end{figure}

\begin{figure}

\includegraphics[scale=.7]{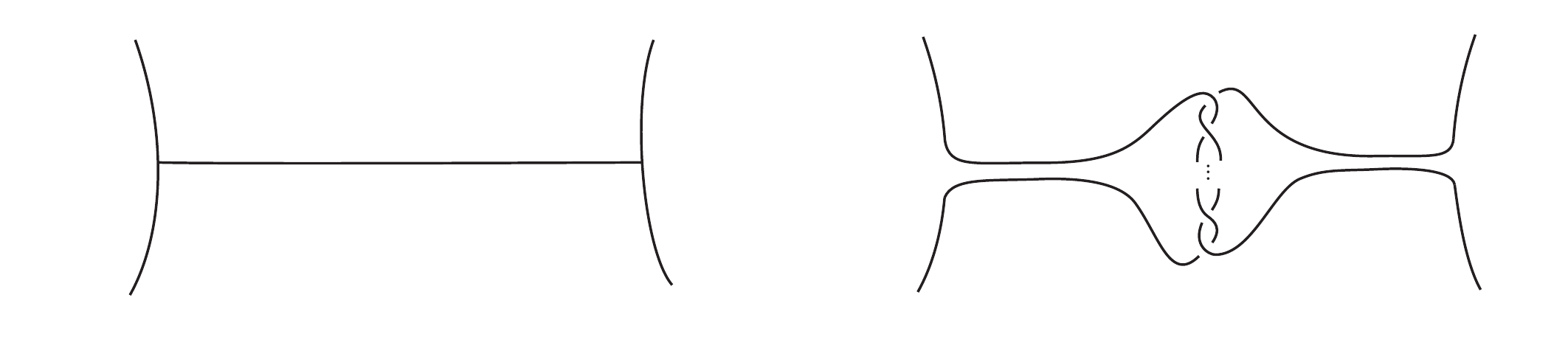}

\caption{Replacing an edge with a twist tangle.}
\label{fig:twist}
\end{figure}
  
  \begin{theorem}
  If two links have isomorphic linking graphs, then they have isomorphic quivers $\mathcal{Q}_{P_n^\sigma}^S$ and equal cocycle $\Phi_\theta$ invariants. 
  \end{theorem}
  
  \begin{proof}
    
 Let $L=K_1\cup \cdots\cup K_m $ be a diagram of an $m$-component link. Let $a_i$ be some arc of $K_i$. The $P_n^\sigma$-colorings of $L$ can be parametrized by the $m$-tuple of colors $(x_1,\dots,x_m)$ assigned to $a_1,\dots,a_m $.  An $m$-tuple of colors $(x_1,\dots, x_m)\in (P_n^\sigma)^m$ represents the colors of the arcs $(a_1,\dots,a_m)$ in a $P_n^\sigma$-colorings of $L$ if and only if whenever $x_i=0$, $|orb_\sigma(x_j)|$ divides $lk(K_i,K_j)$ for all $x_{j}\neq0$. The action of $End(P_n^\sigma)$ on $Col_{P_n^\sigma}(L)$ is determined by its action of the set of $m$-tuples  $(x_1,\dots,x_m)$ representing the colorings. Any other link with an isomorphic linking graph will produce the same set of $m$-tuples  $(x_1,\dots,x_m)$ up to a fixed permutation of the entries. Therefore, their quivers are equal. 
 
If $x_i\neq x_j=0$, then the $\theta$-weight contributed by the crossings of $K_i$ and $K_j$ in the coloring $(x_1,\dots, x_m)$ is $t^{lk(K_i,K_j)}$ as justified in Corollary \ref{cor}.  The $\theta$-weight of a coloring $(x_1,\dots,x_m)$ is
 
 \[ \sum_{i<j \ \land \ (x_i\neq x_j=0 \ \lor \ 0=x_i\neq x_j)  } t^{lk(K_i,K_j)}. \]

\noindent Since the $\theta$-weight of a coloring is determined by its representative $m$-tuple $(x_1,\dots,x_m)$ and $L$'s linking numbers, two links with isomorphic linking graphs will have equal cocycle invariants.
  
  \end{proof}

    \begin{corollary}

If two links have spanning trees in their linking graphs such that no cycle length of $\sigma$ divides a weight in these trees, then their quivers $\mathcal{Q}_{P_n^\sigma}^S$ are isomorphic and their cocycle invariants $\Phi_\theta$ equal $1+n^m$.

\end{corollary}

\begin{proof}

Let $L$ be an $m$-component link with $(x_1,\dots,x_m)\in(P_n^\sigma)^m$ parametrizing its $P_n^\sigma$-colorings. As before, such an $m$-tuple represents a coloring if and only if whenever $x_i=0$, $|orb_\sigma(x_j)|$ divides $lk(K_i,K_j)$ for all $x_j\neq 0$. 

 For any two components $K_i$ and $K_j$ of $L$ consider the path in the spanning tree connecting them. This gives a sequence of components whose consecutive linking numbers are not divisible by any cycle length of $\sigma$. As such, if $K_i$ is colored 0, i.e. $x_i=0$, then each consecutive component in the sequence must also be colored 0, namely $x_j=0$. Therefore, if any entry in $(x_1,\dots,x_m)$ is zero, then all entries are zero.
 
 In addition to the $n^m$ colorings using the positive integers there is a trivial coloring assigning all arcs 0. Each coloring has a trivial $\theta$-weight, so $\Phi_\theta(L)=1+n^m$. Since the action of $End(P_n^\sigma)$ on the tuples $(x_1,\dots,x_m)$ determines the quiver $\mathcal{Q}_{P_n^\sigma}^S(L)$, any other link satisfying the spanning tree condition will have an isomorphic quiver and equal cocycle invariant $\Phi_\theta$.

\end{proof}

  \begin{corollary}
Suppose that $\sigma$ has no fixed points. The pairwise linking numbers of a link with more than 2 components do not determine its quiver $\mathcal{Q}_{P_n^\sigma}^S$ or cocycle invariant $\Phi_\theta$. \label{cor:general}\end{corollary}

  \begin{proof}
  
Take any spanning tree in a complete, simple graph and give all edges a weight of 1. Weight all but one of the remaining edges 1. There are infinitely many choices for the remaining weight that give non-isotopic links with isomorphic quivers and equal cocycle invariants.

  \end{proof}

        \bibliographystyle{amsplain}
\bibliography{proposal.bib}

\end{document}